\documentclass[12pt]{article}

\usepackage{amsmath,amssymb}
\usepackage{amsthm}
\usepackage{pstricks}
\usepackage{graphics,color,url}
\newtheorem{prop}{Proposition}[section]
\newtheorem{thm}[prop]{Theorem}

\newcommand{\sgn}{\text{sgn}}
\newcommand{\inv}{\text{inv}}
\newcommand{\maj}{\text{maj}}

\title{Statistical Distributions and $q$-Analogues of $k$-Fibonacci Numbers}

\author{Adam M. Goyt\footnote{Minnesota State University Moorhead, Moorhead, MN 56563}\\goytadam@mnstate.edu\and Brady L. Keller\footnotemark[1]\\kellerbr@mnstate.edu\and Jonathan E. Rue\footnote{North Dakota State University, Fargo, ND 58102}\\jonathan.rue@ndsu.edu}

\begin{document}

\maketitle

\begin{abstract}
We study $q$-analogues of $k$-Fibonacci numbers that arise from weighted tilings of an $n\times1$ board with tiles of length at most $k$.  The weights on our tilings arise naturally out of distributions of permutations statistics and set partitions statistics.  We use these $q$-analogues to produce $q$-analogues of identities involving $k$-Fibonacci numbers.  This is a natural extension of results of the first author and Sagan on set partitions and the first author and Mathisen on permutations.  In this paper we give general $q$-analogues of $k$-Fibonacci identities for arbitrary weights that depend only on lengths and locations of tiles.  We then determine weights for specific permutation or set partition statistics and use these specific weights and the general identities to produce specific identities.  

\end{abstract}

\section{Introduction}

Miles~\cite{Miles} defines the $k$-generalized Fibonacci numbers $f^{(k)}_n$ by $$f_n^{(k)}=\sum_{i=0}^kf_{n-i}^{(k)},$$ with $f_n^{(k)}=0$ for $0\leq n\leq k-2$ and $f_{k-1}^{(k)}=1$.  

We will work with a shifted version of these {\it $k$-Fibonacci numbers}, $F_n^k$, defined by the recursion $$F_n^k=F_{n-1}^k+F_{n-2}^k+\cdots F_{n-k}^k,$$ if $n>0$, with $F_n^k=0$ if $n<0$ and $F_0^k=1$.  Note that $F_n^2$ are the usual Fibonacci numbers.

An $n\times 1$ {\it board} is a row of $n$ squares as pictured below.  

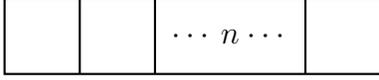
\begin{pspicture}(0,0)(5,1)
\pspolygon(0,0)(0,1)(5,1)(5,0)
\psline(1,0)(1,1)
\psline(2,0)(2,1)
\rput(2.5,.5){$\cdots$}
\rput(3,.5){$n$}
\rput(3.5,.5){$\cdots$}
\psline(4,0)(4,1)
\end{pspicture}

A {\it tile of length $i$} is simply an $i\times 1$ board.  A {\it tiling} is any arrangement of non overlapping tiles that cover every square on an $n\times 1$ board.  We will often restrict the lengths of the tiles that we can use. Let $T_n^k$ be the set of tilings of an $n\times 1$ board using tiles of length at most $k$.  Define $|T_0^k|=1$ by including the empty tiling of the empty board.

\begin{thm}
For $n\in \mathbb{Z}
$, we have $\left|T_n^k\right|=F_n^k$.
\end{thm}

\begin{proof}
Let $A_n^k$ be the number of tilings of an $n\times1$ board using tiles of length at most $k$.
The first tile could be of length $1,2,3,...,$ or $k$.  If the first tile is of length $1$, then all possible tilings of the remaining $(n-1)\times1$ board are counted by $A_{n-1}^k$.  If the first tile is of length $2$, then all possible tilings of the remaining $(n-2)\times1$ board are counted by $A_{n-2}^k$.  This pattern continues until the first tile is of length $k$ in which case the tilings of the remaining $(n-k)\times1$ board are counted by $A_{n-k}^k$.

This gives us that $$A_n^k=A_{n-1}^k+A_{n-2}^k+\dots+A_{n-k}^k.$$

Notice that for $n<0$, $A_n^k=F_n^k=0$ and that $A_0^k=1=F_0^k$.  The former represents the impossible case of a tiling of negative length and the latter represents the empty tiling.  Thus, $\left|T_n^k\right|=F_n^k$ for $n\in \mathbb{Z}$.
\end{proof}

Define a weight function $wt:T_n^k\rightarrow \mathbb{Z}[z_1,z_2,\dots,z_k,q]$ in the following way.  Let $T\in T_n^k$ be a tiling and write $T=t_1t_2\dots t_m$, where $t_j$ is a tile of $T$.  Let $wt(t_j)=z_if_{i,\sigma,\tau}(q)$, where $t_j$ has length $i$, $t_j$ begins in position $\sigma$, the length of the board following $t_j$ is $\tau$, and $f_{i,\sigma,\tau}(q)\in \mathbb{Z}[q]$ is a monomial.  We let $wt$ be a multiplicative function, so that $wt(T)=wt(t_1t_2\dots t_m)=wt(t_1)wt(t_2)\cdots wt(t_m)$.  

If we have a weight function $wt$ as defined above, and a vector ${\bf z}=(z_1,z_2,\dots,z_k)$, then we define $$F_n^k({\bf z};q)=\sum_{\tau\in T_n^k}wt(\tau)$$ to be the $q$-analogue of the $k$-Fibonacci numbers associated to the weight function $wt$.  

In some cases we will only tile a portion of a board at a time and leave either the beginning or the end to be tiled later.  Let $T_{n,m^+}^k$ be the set of tilings of an $n\times 1$ board using tiles of length at most $k$ with an $m\times 1$ board appended to the end, and let $T_{n,m^-}^k$ be the set of tilings of an $n\times 1$ board using tiles of length at most $k$ with an $m\times 1$ board appended to the beginning.  Notice that because we are not tiling the extra $m\times 1$ board, $|T^k_{n,m^+}|=|T^k_{n,m^-}|=F_n^k$.  

This does however change a weighted tiling.  Since the polynomial $f_{i,\sigma,\tau}(q)$ may depend on the length of the board preceding or the length of the board following the tile under consideration, we may need to increase the weight of each tile.  In this case we define the shifted vector $(z_1s^+_{m,1}(q),z_2s^+_{m,2}(q),\dots,z_ks^+_{m,k}(q))$, where the $s^+_{m,i}(q)\in \mathbb{Z}[q]$ are polynomials.  Each of the $s_{m,i}^+(q)$ will increase the weight of a tile of length $i$ by an appropriate amount.  We will write $$\sum_{T\in T_{n,m^+}^k}wt(\tau)=F_n^k(z_1s^+_{m,1}(q),z_2s^+_{m,2}(q),\dots,z_ks^+_{m,k}(q);q)=:F_n^k({\bf zs^+_m};q).$$  Similarly, define the shifted vector $(z_1s^-_{m,1}(q),z_2s^-_{m,2}(q),\dots,z_ks^-_{m,k}(q))$, where the $s_{m,i}^-(q)\in\mathbb{Z}[q]$ are polynomials and write $$\sum_{T\in T_{n,m^-}^k}wt(\tau)=F_n^k(z_1s^-_{m,1}(q),z_2s^-_{m,2}(q),\dots,z_ks^-_{m,k}(q);q)=:F_n^k({\bf zs^-_m};q).$$ 

\begin{thm} For $n\geq 1$ and an arbitrary weight function $wt:T_n^k\rightarrow \mathbb{Z}[z_1,z_2,\dots,z_k,q]$, the $q$-analogue of $F_n^k$ satisfies \begin{eqnarray*}F_n^k({\bf z};q)&=&z_1f_{1,1,n-1}(q)F_{n-1}^k({\bf s^-_1z};q)+z_2f_{2,1,n-2}(q)F_{n-2}^k({\bf s^-_2z};q)+\cdots+\\&&z_kf_{k,1,n-k}(q)F_{n-1}^k({\bf s^-_kz};q).\\ \end{eqnarray*} \end{thm}

\begin{proof}  This follows immediately from the fact that $wt$ is defined to be multiplicative and the proof of Theorem 1.1.  We note that if the first tile is of length $i$ then the remaining tiling is of an $n\times1$ board with an $i\times1$ board appended to the beginning.\end{proof}

In the next section we will use tilings of boards with tiles of length at most $k$ to give bijective proofs of identities involving the $k$-Fibonacci numbers.  Furthermore, we will give identities involving the $q$-analogues of the $k$-Fibonacci numbers in the general sense as we did in Theorem 1.2.  In Section 3, we will describe sets of permutations that are counted by $k$-Fibonacci numbers and provide results involving $q$-analogues where the weight functions are defined by the inversion number or the major index of the permutations.  In Section 4, we will describe sets of  set partitions that are counted by $k$-Fibonacci numbers and provide results involving $q$-analogues where the weight functions are defined by the {\it rb} or the {\it ls} statistic.

\section{Identities}

We will describe identities involving $k$-Fibonacci numbers and provide bijective proofs of these identities using tilings.  We will then adapt these proofs to provide $q$-analogues of these identities using an arbitrary weight function.  

Our first identity is a generalization of a very familiar identity involving Fibonacci numbers.  This identity appears in the paper~\cite{Munarini} of Munarini, but we provide a tiling proof here.  To prove it we need a definition.  Let a {\it break} in a tiling be a place where two tiles come together.  Consider the tiling of a $5\times 1$ board given below.  There are breaks between positions 2 and 3 and between positions 3 and 4 in this tiling.

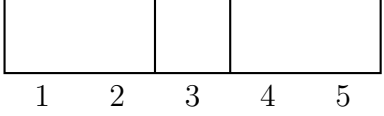
\begin{pspicture}(0,-.5)(5,1)
\pspolygon(0,0)(5,0)(5,1)(0,1)
\psline(2,0)(2,1)
\psline(3,0)(3,1)
\rput(.5,-.3){1}
\rput(1.5,-.3){2}
\rput(2.5,-.3){3}
\rput(3.5,-.3){4}
\rput(4.5,-.3){5}
\end{pspicture}

\begin{thm} For $m\geq1$ and $n\geq1$,
$$F_{m+n}^k=F_m^kF_n^k+\sum_{i=2}^k\sum_{j=1}^{i-1}F_{m-j}^kF_{n-i+j}^k.$$
\end{thm}

\begin{proof}
Consider a tiling of an $(m+n)\times1$ board using tiles of length at most $k$.
Now, there is either a break between the $m^{th}$ and ${m+1}^{st}$ positions or not.  If there is a break, then the number of tilings of the first $m$ positions is $F_m^k$ and the number of tilings of the remaining $n$ positions is $F_n^k$.  So, the total number of tilings if there is a break is $F_m^kF_n^k$.

If there is no break between the $m^{th}$ and ${m+1}^{st}$ positions, there must be one tile covering both of these positions.  This tile can be of length $2,3,\dots,$ or $k$.  If the tile is of length 2, there is only one way it can cover positions $m$ and $m+1$.  The total number of tilings of the board with a tile of length 2 covering the $m^{th}$ and $m+1^{st}$ positions is $F_{m-1}^kF_{n-1}^k.$  

If the tile covering the $m^{th}$ and $m+1^{st}$ position is a tile of length 3 then there are two possible locations for this tile.  If the $m^{th}$ and $m+1^{st}$ positions are covered by the first and second positions in this tile then there are $F_{m-1}^kF_{n-2}^k$ tilings.  If they are covered by the second and third positions then there are $F_{m-2}^kF_{n-1}^k$ tilings.  So, the total number of tilings of the board with a tile of length 3 covering the $m^{th}$ and $m+1^{st}$ positions is $F_{m-1}^kF_{n-2}^k+F_{m-2}^kF_{n-1}^k=\sum_{j=1}^2{F_{m-j}^kF_{n-3+j}}.$

In general, if we consider a tile of length $i$, there are $i-1$ ways for it to cover the $m^{th}$ and $m+1^{st}$ positions, creating a total of $F_{m-1}^kF_{n-(i-1)}^k+F_{m-2}^kF_{n-(i-2)}^k+\dots+ F_{m-(i-1)}^kF_{n-1}^k=\sum_{j=1}^{k-1}{F_{m-j}^kF_{n-i+j}^k}$ tilings.

Summing this last expression over $i$ and adding on the number of tilings with a break between the $m^{th}$ and $m+1^{st}$ positions gives the desired result.\end{proof}

\begin{thm} For $m\geq1$ and $n\geq1$ and an arbitrary weight function $wt$ we have, \begin{eqnarray*}F_{m+n}({\bf z};q)&=&F_m^k({\bf zs_n^+};q)F_n^k({\bf zs_m^-};q)\\&+&\sum_{i=2}^k\sum_{j=1}^{i-1}z_if_{i,m-j+1,n-i+j}(q)F_{m-j}^k({\bf zs_{n+j}^+};q)F^k_{n-i+j}({\bf zs^-_{m+i-j}};q).\end{eqnarray*}\end{thm}
\begin{proof}

This proof is very similar to the one above, we simply consider the weights of the tiles we are using.  

If there is a break between the $m^{th}$ and $m+1^{st}$ tile then the $m\times 1$ board at the beginning has an $n\times 1$ board appended to the end.  Thus the number of weighted tilings of the first $m$ positions is $F_m^k({\bf zs_n^+};q)$.  Similarly the number of weighted tilings of the final $n$ positions is $F_n^k(zs^-_m;q)$.  Thus, the number of weighted tilings of an $(m+n)\times1$ board where there is a break between the $m^{th}$ and $m+1^{st}$ positions is $F_m^k({\bf zs_n^+};q)F_n^k({\bf zs_m^-};q)$.  

Suppose there is a tile of length $i$ covering the $m^{th}$ and $m+1^{st}$ positions.  Furthermore, assume that the $j^{th}$ position of this tile covers the $m^{th}$ position of the larger board.  This tile has weight $z_if_{i,m-j+1,n-i+j}(q)$ because it is of length $i$, it begins in position $m-j+i$ and there is a board of length $n-i+j$ following it.  The number of weighted tilings of the $m-j$ positions preceding this tile is $F_{m-j}^k({\bf zs^+_{n+j}};q)$ since each of these is a weighted tiling of an $(m-j)\times 1$ board with an $(n+j)\times1$ board appended to the end.  Similarly, the number of weighted tilings of the $n-i+j$ positions following the tile of length $i$ is $F_{n-i+j}^k({\bf zs_{m+i-j}};q)$. This gives  $$\sum_{i=2}^k\sum_{j=1}^{i-1}z_if_{i,m-j+1,n-i+j}(q)F_{m-j}^k({\bf zs_{n+j}^+};q)F^k_{n-i+j}({\bf zs^-_{m+i-j}};q)$$ weighted tilings of a board without a break between the $m^{th}$ and $m+1^{st}$ positions.

Summing this last expression over $i$ and adding on the number of tilings with a break between the $m^{th}$ and $m+1^{st}$ positions gives the desired result.\end{proof}

\begin{thm} For $n\geq1$,
$$F^k_n=F_n^{k-1}+\sum_{j=0}^{n-k}F_j^{k-1}F_{n-k-j}^k.$$
\end{thm}
\begin{proof}
Consider a tiling of an $n\times1$ board consisting of tiles of length at most $k$.  Now, there is either a tile of length $k$ in the tiling or not.  If there is no tile of length $k$ in the tiling then there are $F_n^{k-1}$ tilings.  

Suppose there is a tile of length $k$ in the tiling.  Consider the location of the first tile of length $k$.  If the first tile of length $k$ begins at the $j+1^{st}$ position on the board then the first $j$ positions must be tiled with tiles of length at most $k-1$, which can be done in $F_j^{k-1}$ ways.  The final $n-k-j$ positions can be tiled with tiles of length at most $k$, which can be done in $F_{n-k-j}^k$ ways.  Summing over appropriate values of $j$ gives $F^k_n=F_n^{k-1}+\sum_{j=0}^{n-k}F_j^{k-1}F_{n-k-j}^k$.
\end{proof}

\begin{thm}  For $n\geq 1$ and an arbitrary weight function $wt$,
$$F_n^k({\bf z};q)=F_n^{k-1}({\bf z};q)+\sum_{j=0}^{n-k}z_kf_{k,j+1,n-k-j}(q)F_j^{k-1}({\bf zs^+_{n-j}};q)F_{n-k-j}^k({\bf zs^-_{k+j}};q).$$\end{thm}

\begin{proof}  This follows directly from the definitions, and the proof of Theorem 2.3.  \end{proof}

Miles~\cite{Miles} proved the following determinant identity involving $k$-generalized Fibonacci numbers.  $$\det \left(\begin{array}{cccc} f_n^{(k)}&f_{n+1}^{(k)}&\cdots&f_{n+k-1}^{(k)}\\ f_{n+1}^{(k)}&f_{n+2}^{(k)}&\cdots&f_{n+k}^{(k)}\\ \vdots&\vdots&\ddots&\vdots\\ f_{n+k-1}^{(k)}&f_{n+k}^{(k)}&\cdots&f_{n+2k-2}^{(k)}\\ \end{array}\right)=(-1)^{\frac{(2n+k)(k-1)}{2}}.$$

We will use a method of lattice paths and their relationships to minors of a Toeplitz-like matrix for $k$-Fibonacci numbers to prove a weighted version of this theorem for $F_n^k$.  Lindstr\"{o}m~\cite{Lindstrom} introduced this method, and Gessel and Viennot~\cite{GesselViennot} showed that it has broad application.  This method was applied in \cite{SagGoytqFib} to determine a family of $q$-analogues of Fibonacci numbers.  We will describe it in it's entirety here and use it to determine a $q$-analogue of the identity $$\det \left(\begin{array}{cccc}F_{n+k-1}^k&F_{n+k}^k&\cdots&F_{n+2k-2}^k\\ F_{n+k-2}^k&F_{n+k-1}^k&\cdots&F_{n+2k-3}^k\\ \vdots&\vdots&\ddots&\vdots\\F_n^k&F_{n+1}^k&\cdots&F_{n+k-1}^k\\ \end{array}\right)=\left\{\begin{array}{cl}1&\mbox{ if }k\mbox{ is odd,}\\ (-1)^{n-1}&\mbox{ if }k\mbox{ is even.}\end{array}\right.$$  The proof of this identity is an immediate consequence of the proof of Theorem 2.5.

Consider the digraph $D_k=(V_k,A_k)$ with vertices labeled $0$, $1$, $2\dots$ and for each vertex $n$ there are $k$ arcs beginning at $n$ and ending at $n+1$, $n+2$, $\dots$, $n+k$ respectively.  It is easy to see that the number of directed walks from vertex $a$ to vertex $b$ is $F_{b-a}^k$.  Below is the portion of the digraph on the vertices $0,1,2,\dots,7$ for $k=3$.  All arcs are directed to the right.

\noindent
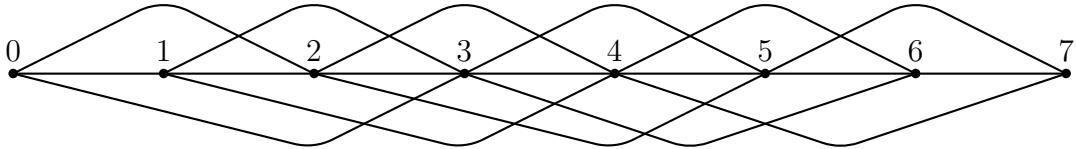
\begin{pspicture}(1,-2)(15,2)
\psdots(1,0)(3,0)(5,0)(7,0)(9,0)(11,0)(13,0)(15,0)
\psline(1,0)(15,0)
\psline[linearc=.75](1,0)(3,1)(5,0)
\psline[linearc=.75](3,0)(5,1)(7,0)
\psline[linearc=.75](5,0)(7,1)(9,0)
\psline[linearc=.75](7,0)(9,1)(11,0)
\psline[linearc=.75](9,0)(11,1)(13,0)
\psline[linearc=.75](11,0)(13,1)(15,0)
\psline[linearc=.75](1,0)(5,-1)(7,0)
\psline[linearc=.75](3,0)(7,-1)(9,0)
\psline[linearc=.75](5,0)(9,-1)(11,0)
\psline[linearc=.75](7,0)(10,-1)(13,0)
\psline[linearc=.75](9,0)(12,-1)(15,0)
\rput(1,.3){$0$}
\rput(3,.3){$1$}
\rput(5,.3){$2$}
\rput(7,.3){$3$}
\rput(9,.3){$4$}
\rput(11,.3){$5$}
\rput(13,.3){$6$}
\rput(15,.3){$7$}
\end{pspicture}

Let the edge from $n$ to $n+i$ for $1\leq i\leq k$ be written $\vec{e}_{n,n+i}$ and let an arbitrary weight $wt(\vec{e}_{n,n+i})$ be the same as the corresponding weight of a tile of length $i$ in a tiling.  Let $p$ be a directed path from vertex $a$ to vertex $b$ and let $wt(p)$ be the product of the weights of its arcs.  We have that $$\sum_{p} wt(p)=F_{b-a}^k({\bf s^-_az};q),$$ where the sum is over all paths $p$ from $a$ to $b$.  We will not consider the portion of the graph after $b$, so we do not need to include any term of the form $s_{m,i}^+(q)$.

Let ${\bf u}:u_1<u_2<\cdots<u_m$ and ${\bf v}:v_1<v_2<\cdots<v_m$ be sequences of vertices in $D_k$.  An $m$-tuple of paths from $\bf{u}$ to $\bf{v}$ is $$P=\left\{u_1\stackrel{p_1}{\rightarrow}v_{\alpha(1)},u_2\stackrel{p_2}{\rightarrow}v_{\alpha(2)},\dots,u_m\stackrel{p_m}{\rightarrow}v_{\alpha(m)}\right\},$$ where $\alpha\in S_m$, the symmetric group on $m$ elements.  Define the weight of such an $m$-tuple to be $wt(P)=\prod_{i=1}^mwt(p_i)$.  Let the sign of the $m$-tuple of paths be $\sgn(P)=\sgn(\alpha)$.

Now, the matrix above is the minor of the Toeplitz-like matrix $$\left(\begin{array}{cccc}F_0^k&F_1^k&F_2^k&\cdots\\ 0&F_0^k&F_1^k&\cdots\\ 0&0&F_0^k&\cdots\\ \vdots&\vdots&\vdots&\ddots\end{array}\right)$$ given by the first $k$ rows and columns $n+k-1$ through $n+2k-2$.  Since we are interested in the $q$-analogue of the identity above, we will be interested in the Toeplitz-like matrix $$\hat{F}=\left(\begin{array}{cccc}F_0^k(\bf{z};q)&F_1^k(\bf{z};q)&F_2^k(\bf{z};q)&\cdots\\ 0&F_0^k(\bf{s_1z};q)&F_1^k(\bf{s_1z};q)&\cdots\\ 0&0&F_0^k(\bf{s_2z};q)&\cdots\\ \vdots&\vdots&\vdots&\ddots\end{array}\right).$$  Let $\hat{F}_{\bf{u},\bf{v}}$ be the minor of this matrix given by the rows indexed by the sequence $\bf{u}$ and columns indexed by the sequence $\bf{v}$.  From the weight definitions above we have that $$\det \hat{F}_{\bf{u},\bf{v}}=\sum_{P}\sgn(P)wt(P),$$ where the sum is over all $m$-tuples of paths from $\bf{u}$ to $\bf{v}$.  

The minor we are interested in is the one with \begin{equation}\label{bob}{\bf u}=0,1,\dots,k-1 \mbox{ and } {\bf v}=n+k-1,n+k,\dots,n+2k-2.\end{equation}  Our goal is describe a simple way to compute the determinant of this minor. We will say that an $m$-tuple of paths is {\it noncrossing} if no two paths share a vertex.

Theorem 5.2 from ~\cite{SagGoytqFib} shows that $$\det \hat{F}_{\bf{u},\bf{v}}=\sum_{P}\sgn(P)wt(P),$$ where the sum is over all $m$-tuples of noncrossing paths.  This is sufficient to compute the determinant of $\hat{F}_{{\bf u},{\bf v}}$.  We observe that because the vertices ${\bf u}:0,1,\dots,k-1$ are adjacent and the vertices ${\bf v}:n+k-1,n+k,\dots,n+2k-2$ are adjacent, the only $k$-tuple of paths from $ \bf{u}$ to $\bf{v}$ are paths that consist entirely of the edges $\vec{e}_{i,i+k}$ for $0\leq i\leq n+k-2$.  Let $P_{n}^k$ be this $k$-tuple of paths.  

The $k$-tuple of paths $P_n^k$ consists entirely of arcs of length $k$.  We need to determine the weights of these arcs in the context of  $z_kf_{k,\sigma,\tau}(q)$.  Suppose that $n+k-1=pk+r$ where $p$ and $r$ are nonnegative integers and $r<k$.  The first $r$ arcs are followed by paths of length $pk$.  Thus the product of the weights of these arcs is $\prod_{i=1}^rz_kf_{k,i,pk}(q)$.  Now, the next $k$ arcs are each followed by a path of length $(p-1)k$, so the weights of these arcs is $\prod_{i=r+1}^{r+k}z_kf_{k,i,(p-1)k}(q)$.  This process continues until we get to the last $k$ arcs which are followed by paths of length 0.  Thus, the weights of the arcs are $$z_k^{n+k-1}\prod_{i=1}^rf_{k,i,pk}(q)\prod_{j=0}^{p-1}\prod_{a=1}^kf_{k,r+jk+a,(p-j-1)k}(q).$$  

Due to the nature of the paths in $P^n_k$ the associated permutation $\alpha$ must be of the form $r(r+1)\dots k 12\dots (r-1)$, where $1\leq r\leq k$.  This permutation has $(r-1)(k-(r-1))$ inversions.  Thus, if $k$ is odd, $\sgn(\alpha)$ is even.  

If $k$ is even, then the parity of $\alpha$ is the same as the parity of $r-1$.  Now, suppose that $n+k-1$ is divisible by $k$.  This implies that $n-1$ is divisible by $k$ and hence is even.  Also, if $n+k-1$ is divisible by $k$ then $\alpha=123\dots k$, i.e. $r=1$.  Thus, $\alpha$ is even when $n-1$ is even.

This gives us the following theorem.

\begin{thm}  Let ${\bf u}$ and ${\bf v}$ be as in \eqref{bob} and $n+k-1=pk+r$ where $p$ and $r$ are nonnegative integers and $r<k$, then we have that $$\det \hat{F}_{\bf{u},\bf{v}}=\left\{\begin{array}{lc}z_k^{n+k-1}\prod_{i=1}^rf_{k,i,pk}(q)\prod_{j=0}^{p-1}\prod_{a=1}^kf_{k,r+jk+a,(p-j-1)k}(q)&\mbox{ if } k \mbox{ is odd,}\\ (-1)^{n-1}z_k^{n+k-1}\prod_{i=1}^rf_{k,i,pk}(q)\prod_{j=0}^{p-1}\prod_{a=1}^kf_{k,r+jk+a,(p-j-1)k}(q)&\mbox{ if } k \mbox{ is even.}\end{array}\right.$$\end{thm} 

In the next sections we will give specific $q$-analogues of the identities listed below by determining distributions of statistics over certain combinatorial objects.  We will show that the distributions depend only on the ``tiles'', which will allow us to determine $f_{i,\sigma,\tau}(q)$, $s^-_{m,i}(q)$, and $s^+_{m,i}(q)$ for each of these sets and statistics.  We will then be able to simply substitute into the identities determined above to give $q$-analogues.  

Because we will be determining $q$-analogues using many different statistics and many different sets, we will call all of them $F_n^k({\bf z};q)$.  We hope that the context maintains the clarity of which $q$-analogue we are discussing.  

\begin{eqnarray*}
F_n^k({\bf z};q)&=&z_1f_{1,1,n-1}(q)F_{n-1}^k({\bf s^-_1z};q)+z_2f_{2,1,n-2}(q)F_{n-2}^k({\bf s^-_2z};q)+\cdots+\\&&z_kf_{k,1,n-k}(q)F_{n-1}^k({\bf s^-_kz};q).\\
F_{m+n}({\bf z};q)&=&F_m^k({\bf zs_n^+};q)F_n^k({\bf zs_m^-};q)+\\ &&\sum_{i=2}^k\sum_{j=1}^{i-1}z_if_{i,m-j+1,n-i+j}(q)F_{m-j}^k({\bf zs_{n+j}^+};q)F^k_{n-i+j}({\bf zs^-_{m+i-j}};q).\\
F_n^k({\bf z};q)&=&F_n^{k-1}({\bf z};q)+\sum_{j=0}^{n-k}z_kf_{k,j+1,n-k-j}(q)F_j^{k-1}({\bf zs^+_{n-j}};q)F_{n-k-j}^k({\bf zs^-_{k+j}};q).\\
\det\hat{F}_{\bf{u},\bf{v}}&=&\left\{\begin{array}{ll}z_k^{n+k-1}\prod_{i=1}^rf_{k,i,pk}(q)\prod_{j=0}^{p-1}\prod_{a=1}^kf_{k,r+jk+a,(p-j-1)k}(q)&k \mbox{ odd,}\\ (-1)^{n-1}z_k^{n+k-1}\prod_{i=1}^rf_{k,i,pk}(q)\prod_{j=0}^{p-1}\prod_{a=1}^kf_{k,r+jk+a,(p-j-1)k}(q)&k \mbox{ even.}\end{array}\right.
\end{eqnarray*}

\section{Permutations and $q$-Analogues}

We will consider distributions of permutation statistics over certain sets of permutations that are counted by the $k$-Fibonacci numbers.  These distributions will give rise to specific $q$-analogues of the identities discussed above.  For each statistic we will only need to determine the weight of a tile and the shift factor and substitute these into the identities from Section 2.  The sets of permutations we will discuss arise naturally from the study of pattern avoidance.  We will discuss the structure of the permutations in these sets and refer the reader to papers in which these results appear, but we will not discuss the idea of pattern avoidance.  We refer the reader to~\cite{Wilfpatterns} for an overview of pattern avoidance.  

We define {\it layered permutations} to be permutations of the form $$m(m+1)\dots n j(j+1)\dots (m-1)\dots 12\dots (i-1).$$  Let $LP_n$ be the set of layered permutations in $S_n$, then for example $$LP_4=\{1234,2341,3412,3421,4123,4231,4312,4321\}.$$  Each contiguous increasing sequence in a layered permutation is called a {\it layer}.  For example the permutation $4231$ has three layers, $4$, $23$, and $1$.  We will say that the {\it length} of a layer is the number of elements in the layer. 

Let $r:S_n\rightarrow S_n$ be the bijective map that satisfies, $r(\pi_1\pi_2\dots\pi_n)=\pi_n\pi_{n-1}\dots\pi_1$, where $\pi_1\pi_2\dots\pi_n\in S_n$.  We call $r$ the {\it reversal} map.  Let $RLP_n=\{r(\pi):\pi\in LP_n\}$ be the set of {\it reverse layered permutations}.  The layers of a permutation in $RLP_n$ are the reverse of the layers of the corresponding permutation in $LP_n$.  For example, $RLP_4=\{4321,1432,2143,1243,3214,1324,2134,1234\}$.

Let $PRLP_n$ be the set of permutations obtained from $LP_n$ be reversing all but the last element in each layer of a permutation $\pi\in LP_n$.  We will call these {\it partially reversed layered permutations}.  We have $PRLP_4=\{3214,3241,3412,3421,4213,4231,4312,4321\}$.

Let $LP^k_n$ ($RLP_n^k$, $PRLP_n^k$) be the sets of layered (reverse layered, partially reverse layered) permutations with layers of length at most $k$.  Mansour~\cite{MansourkFib} proves that $|LP_n^k|=|RLP_n^k|=|PRLP_n^k|=F_n^k$.  In~\cite{GoytMathisen} the first author and Mathisen study statistical distributions over $RLP_n^2$.

We observe that any element of $LP_n^k$, $RLP_n^k$, or $PRLP_n^k$ is uniquely determined by its layers, and these layers can be thought of as tiles in a tiling of an $n\times 1$ board.  This gives an obvious bijection with $T_n^k$.  In the following subsections we will discuss specific permutation statistics and their associated weight functions for $LP_n^k$, $RLP_n^k$, and $PRLP_n^k$, and use them to give specific $q$-analogues of the identities from Section 2.

\subsection{The Inversion Statistic}

Let $\pi\in S_n$.  We say $\pi(i)$ and $\pi(j)$ form an inversion if $i<j$ and $\pi(i)>\pi(j)$.  Let $\inv(\pi)$ be the number of inversions in $\pi$.  For example, the permutation $\pi=453612$ has $\inv(\pi)=10$.  

\subsubsection{ $LP_n^k$}

Consider the set $LP_n^k$.  Let $\pi=\pi_1\pi_2\dots\pi_m\in LP_n^k$ where $\pi_j$ is a layer for $1\leq j\leq m$.  Then any element in the layer $\pi_j$ is larger than every element in a layer following $\pi_j$.  Furthermore, the elements of $\pi_j$ are in increasing order.  Thus, there are no inversions inside of the layer $\pi_j$.  If $\pi_j$ has length $i$, denoted $\ell(\pi_j)=i$ and the length of the permutation following $\pi_j$ is $\ell(\pi_{j+1}\dots\pi_m)$.  Then we will say the weight of $\pi_j$ is $wt(\pi_j)=z_iq^{i\ell(\pi_{j+1}\dots\pi_m)}$.  In the context of the general weight functions we would say that $f_{i,\sigma,\tau}(q)=q^{i\ell(\pi_{j+1}\dots\pi_m)}=q^{i\tau}$.  From the definition of $\inv(\pi)$, we see that $wt(\pi)=\left(\prod_{j=1}^mz_{\ell(\pi_j)}\right)q^{\inv(\pi)}$.  Define $F_n^k({\bf z};q)=\sum_{\pi\in LP_n^k}wt(\pi)$.   

We need only determine the shift factors $s_{m,i}^+(q)$ and $s_{m,i}^-(q)$ for this set and this statistic.  Since the weight of a layer only depends on the length of the permutation appearing after this layer, we must have that for all $i$ and $m$, $s_{m,i}^-(q)=1$ and $s_{m,i}^+(q)=q^{im}$.  Thus $F_n^k({\bf zs_m^+};q)=F_n^k(z_1q^m,z_2q^{2m},\dots,z_kq^{km})$, which essentially means that the weight of each position is increased by a factor of $q^m$.  Since there are $n$ positions, we have $$F_n^k({\bf zs^+_m};q)=q^{nm}F_n({\bf z};q).$$

For $n\geq1$, the $q$-analogue of the $k$-Fibonacci numbers for the distribution of the inversion statistic over $LP_n^k$ satisfies $$F_n^k({\bf z};q)=z_1q^{n-1}F_{n-1}^k({\bf z};q)+z_2q^{2(n-2)}F_{n-2}({\bf z};q)+\dots+z_kq^{k(n-k)}F_{n-k}({\bf z};q),$$ where $F_0^k({\bf z};q)=1$ and $F_n^k({\bf z};q)=0$ if $n<0$.  

Given the weight and shifting factor associated to the distribution of the inversion statistic over $LP_n^k$ we obtain the following identities for ${\bf u}$ and ${\bf v}$ as in \eqref{bob} and $n+k-1=pk+r$ with $p$ and $r$ nonnegative integers and $r<k$: 

\begin{eqnarray*}
F_{m+n}({\bf z};q)&=&q^{nm}F_m^k({\bf z};q)F_n^k({\bf z};q)+\sum_{i=2}^k\sum_{j=1}^{i-1}z_iq^{i(n+j)}q^{(m-j)(n+j)}F_{m-j}^k({\bf z};q)F_{n-i+j}({\bf z};q).\\
F_n^k({\bf z};q)&=&F_n^{k-1}({\bf z};q)+\sum_{j=0}^{n-k}z_kq^{k(n-k-j)}q^{j(n-j)}F_j^{k-1}({\bf z};q)F_{n-k-j}^k({\bf z};q).\\
\det\hat{F}_{\bf{u},\bf{v}}&=&\left\{\begin{array}{lc}z_k^{n+k-1}q^{prk^2+k^3\binom{p}{2}}&\mbox{ if } k \mbox{ is odd,}\\ (-1)^{n-1}z_k^{n+k-1}q^{prk^2+k^3\binom{p}{2}}&\mbox{ if } k \mbox{ is even.}\end{array}\right. 
\end{eqnarray*}

\subsubsection{$RLP^k_n$}

Consider the set $RLP^k_n$.  Let $\pi=\pi_1\pi_2\dots\pi_m\in RLP_n^k$, where each $\pi_j$ is a layer.  Since all of the elements of $\pi_i$ are less than all of the elements of $\pi_j$ if $i<j$, we need only look inside each layer to determine the number of inversions.  The elements of any layer $\pi_j$ are in decreasing order, so if $\ell(\pi_j)=i$ then there are $\binom{i}{2}$ inversions.  This gives us that $f_{i,\sigma,\tau}(q)=q^{\binom{i}{2}}$ in this case, so $wt(\pi_j)=z_iq^{\binom{i}{2}}$, and hence $wt(\pi)=\left(\prod_{j=1}^mz_{\ell(\pi_j)}\right)q^{\inv(\pi)}$.  Define $F_n^k({\bf z};q)=\sum_{\pi\in RLP_n^k}wt(\pi)$.

Finally, we need to determine the shifting factors.  Since the number of inversions does not depend on the location of the layer, we have that $s^-_{m,i}=1$ and $s^+_{m,i}=1$ for each $m$ and $i$.  

For $n\geq1$, the $q$-analogue of the $k$-Fibonacci numbers for the distribution of the inversion statistic over $RLP_n^k$ satisfies $$F_n^k({\bf z};q)=z_1F_{n-1}^k({\bf z};q)+z_2qF_{n-2}({\bf z};q)+\dots+z_kq^{\binom{k}{2}}F_{n-k}({\bf z};q),$$ where $F_0^k({\bf z};q)=1$ and $F_n^k({\bf z};q)=0$ if $n<0$. 

Given the weight and shifting factor associated to the distribution of the inversion statistic over $RLP_n^k$ we obtain the following identities for ${\bf u}$ and ${\bf v}$ as in \eqref{bob}: 

\begin{eqnarray*}
F_{m+n}({\bf z};q)&=&F_m^k({\bf z};q)F_n^k({\bf z};q)+\sum_{i=2}^k\sum_{j=1}^{i-1}z_iq^{\binom{i}{2}}F_{m-j}^k({\bf z};q)F_{n-i+j}^k({\bf z};q).\\
F_n^k({\bf z};q)&=&F_n^{k-1}({\bf z};q)+\sum_{j=0}^{n-k}z_kq^{\binom{k}{2}}F_j^{k-1}({\bf z};q)F_{n-k-j}^k({\bf z};q).\\
\det \hat{F}_{{\bf u},{\bf v}}&=&\left\{\begin{array}{lc}z_k^{n+k-1}q^{(n+k-1)\binom{k}{2}}&\mbox{ if } k \mbox{ is odd,}\\ (-1)^{n-1}z_k^{n+k-1}q^{(n+k-1)\binom{k}{2}}&\mbox{ if } k \mbox{ is even.}\end{array}\right.
\end{eqnarray*}

\subsubsection{$PRLP_n^k$}

Let $\pi=\pi_1\pi_2\dots\pi_m\in PRLP_n^k$, where each $\pi_j$ is a layer.  Suppose $\pi_j$ is a layer of length $i$.  The first $i-1$ elements are in decreasing order now, with the largest element in the last position.  This gives $\binom{i-1}{2}$ inversions inside of each layer of length $i$.  Since the layers in $RPLP_n^k$ are in the same order as they were in $LP_n^k$ each element of $\pi_j$ is larger than any element in a layer following $\pi_j$.  This gives us that $f_{i,\sigma,\tau}(q)=q^{\binom{i-1}{2}+i\tau}$.  Hence, $wt(\pi_j)=z_iq^{\binom{i-1}{2}+i\tau}$ and $wt(\pi)=\left(\prod_{j=1}^mz_{\ell(\pi_j)}\right)q^{\inv(\pi)}$.  Define $F_n^k({\bf z};q)=\sum_{\pi\in PRLP_n^k}wt(\pi)$.  

As in the case of $LP_n^k$, we have that $s^+_{m,i}(q)=q^{im}$ and $s^-_{m,i}(q)=1$.  Thus $F_n^k({\bf zs_m^+};q)=F_n^k(z_1q^m,z_2q^{2m},\dots,z_kq^{km})=F_n^k({\bf zs^+_m};q)=q^{nm}F_n({\bf z};q).$

For $n\geq1$, the $q$-analogue of the $k$-Fibonacci numbers for the distribution of the inversion statistic over $PRLP_n^k$ satisfies $$F_n^k({\bf z};q)=z_1q^{n-1}F_{n-1}^k({\bf z};q)+z_2q^{2(n-2)}F_{n-2}({\bf z};q)+\dots+z_kq^{\binom{k-1}{2}+k(n-k)}F_{n-k}({\bf z};q),$$ where $F_0^k({\bf z};q)=1$ and $F_n^k({\bf z};q)=0$ if $n<0$.  

Given the weight and shifting factor associated to the distribution of the inversion statistic over $PRLP_n^k$ we obtain the following identities for ${\bf u}$ and ${\bf v}$ as in \eqref{bob} and $n+k-1=pk+r$ with $p$ and $r$ nonnegative integers and $r<k$: 

\begin{eqnarray*}
F_{m+n}({\bf z};q)&=&q^{nm}F_m^k({\bf z};q)F_n^k({\bf z};q)+\\&&\sum_{i=2}^k\sum_{j=1}^{i-1}z_iq^{\binom{i-1}{2}+i(n-i+j)}q^{(m-j)(n+j)}F_{m-j}^k({\bf z};q)F_{n-i+j}({\bf z};q).\\
F_n^k({\bf z};q)&=&F_n^{k-1}({\bf z};q)+\sum_{j=0}^{n-k}z_kq^{\binom{k-1}{2}+k(n-k-j)}q^{j(n-j)}F_j^{k-1}({\bf z};q)F_{n-k-j}^k({\bf z};q).\\
\det\hat{F}_{\bf{u},\bf{v}}&=&\left\{\begin{array}{lc}z_k^{n+k-1}q^{(n+k-1)\binom{k-1}{2}+prk^2+k^3\binom{p}{2}}&\mbox{ if } k \mbox{ is odd,}\\ (-1)^{n-1}z_k^{n+k-1}q^{(n+k-1)\binom{k-1}{2}+prk^2+k^3\binom{p}{2}}&\mbox{ if } k \mbox{ is even.}\end{array}\right. 
\end{eqnarray*}

\subsection{The Major Index}

Let $\pi\in S_n$.  We say that $\pi(i)$ and $\pi(i+1)$ form a {\it descent}  if $\pi(i)>\pi(i+1)$.  Define $D_\pi=\{i:\pi(i)\pi(i+1)\mbox{ is a descent}\}$.  We define the {\it major index} of $\pi$ to be $\maj(\pi)=\sum_{i\in D_\pi}i$.  For example, the permutation $\pi=453612$ has $\maj(\pi)=2+4=6$.

\subsubsection{$LP_n^k$}

Let $\pi=\pi_1\pi_2\dots\pi_m\in LP_n^k$, where $\pi_j$ is a layer for $1\leq j\leq m$.  The elements in the layers of $\pi$ are increasing, so there are no descents within layers.  The last element in the first $m-1$ layers is larger than the first element in the next layer.  Thus, in all but the last layer, there is a descent from the last element in the layer to the first element in the next layer.  For the major index we sum the locations of the the tops of the descents, which are the locations of the ends of all but the last layer.  Suppose $\pi_j$ is a layer of $\pi$ of length $i$.  Define $wt(\pi_j)=z_if_{i,\sigma,\tau}(q)=z_iq^{\sigma-1}$.  The term $q^{\sigma-1}$ counts the index of the descent starting in the layer preceding $\pi_j$.  Thus, $wt(\pi)=\left(\prod_{j=1}^mz_{\ell(\pi_j)}\right)q^{\maj(\pi)}$.  We define $$\sum_{\pi\in LP_n^k}wt(\pi)=F_n^k({\bf z};q).$$

We need only determine the shifting factors associated to this statistic.  Since $f_{i,\sigma,\tau}(q)$ depends only on the length of the permutation preceding the layer, we have $s_{m,i}^+(q)=1$ and $s_{m,i}^-(q)=q^m$.  Thus, ${\bf z s_m^-(q)}=(z_1q^m,z_2q^m,\dots,z_kq^m)$.  For simplicity, we will write ${\bf zs_m^-(q)}={\bf zq^m}$.  

For $n\geq1$, the $q$-analogue of the $k$-Fibonacci numbers for the distribution of the major index over $LP_n^k$ satisfies $$F_n^k({\bf z};q)=z_1F_{n-1}^k({\bf zq};q)+z_2F_{n-2}({\bf zq^2};q)+\dots+z_kF_{n-k}({\bf zq^k};q),$$ where $F_0^k({\bf z};q)=1$ and $F_n^k({\bf z};q)=0$ if $n<0$.  

Given the weight and shifting factor associated to the distribution of the major index over $LP_n^k$ we obtain the following identities for ${\bf u}$ and ${\bf v}$ as in \eqref{bob}:

\begin{eqnarray*}
F_{m+n}({\bf z};q)&=&F_m^k({\bf z};q)F_n^k({\bf zq^m};q)+\sum_{i=2}^k\sum_{j=1}^{i-1}z_iq^{m-j}F_{m-j}^k({\bf z};q)F_{n-i+j}^k({\bf zq^{m+i-j}};q).\\
F_n^k({\bf z};q)&=&F_n^{k-1}({\bf z};q)+\sum_{j=0}^{n-k}z_kq^{j}F_j^{k-1}({\bf z};q)F_{n-k-j}^k({\bf zq^{k+j}};q).\\
\det \hat{F}_{{\bf u},{\bf v}}&=&\left\{\begin{array}{lc}z_k^{n+k-1}q^{\binom{n+k-1}{2}}&\mbox{ if } k \mbox{ is odd,}\\ (-1)^{n-1}z_k^{n+k-1}q^{\binom{n+k-1}{2}}&\mbox{ if } k \mbox{ is even.}\end{array}\right.
\end{eqnarray*}

\subsubsection{$RLP_n^k$}

Let $\pi=\pi_1\pi_2\dots\pi_m\in RLP_n^k$, where $\pi_j$ is a layer for $1\leq j\leq m$.  The elements in each layer of $\pi$ are decreasing, and hence all but the last position of each layer is the top of a descent.  The last position of each layer is smaller than the first position of the next layer, so there are no descents between layers.  We are adding the positions of each descent to determine $\maj(\pi)$.  Consider the layer $\pi_j$, assume $\ell(\pi_j)=i$, and that $\pi_j$ begins in position $\sigma$.  Then the first $i-1$ elements of $\pi_j$ are in positions $\sigma$, $\sigma+1$, $\dots$, $\sigma+i-2$.  Summing these gives $(\sigma-1)(i-1)+\binom{i}{2}$.  Let $wt(\pi_j)=z_if_{i,\sigma,\tau}(q)=z_iq^{(\sigma-1)(i-1)+\binom{i}{2}}$.  This gives us that $wt(\pi)=\left(\prod_{j=1}^mz_{\ell(\pi_j)}\right)q^{\maj(\pi)}$.  We define $F_n^k({\bf z};q)=\sum_{\pi\in RLP_n^k}wt(\pi).$

We need only determine the shifting factors for this distribution.  Again $f_{i,\sigma,\tau}(q)$ only depends on the length of the permutation preceding the layer, we have that $s_{m,i}^+(q)=1$ and $s_{m,i}^-(q)=q^{(i-1)m}$ for each $i$ and $m$.  We have ${\bf z s_m^-(q)}=(z_1,z_2q^m,z_3q^{2m}\dots,z_kq^{(k-1)m})$.  For simplicity, we will write ${\bf zs_m^-(q)}={\bf zq^{(i-1)m}}$. 

For $n\geq1$, the $q$-analogue of the $k$-Fibonacci numbers for the distribution of the major index over $RLP_n^k$ satisfies $$F_n^k({\bf z};q)=z_1F_{n-1}^k({\bf zq^{(i-1)}};q)+z_2q^{}F_{n-2}({\bf zq^{(i-1)2}};q)+\cdots+z_kq^{\binom{k}{2}}F_{n-k}({\bf zq^{(i-1)k}};q),$$ where $F_0^k({\bf z};q)=1$ and $F_n^k({\bf z};q)=0$ if $n<0$.  

Given the weight and shifting factor associated to the distribution of the major index over $LP_n^k$ we obtain the following identities for ${\bf u}$ and ${\bf v}$ as in \eqref{bob}:

\begin{eqnarray*}
F_{m+n}({\bf z};q)&=&F_m^k({\bf z};q)F_n^k({\bf zq^{(i-1)m}};q)+\\&&\sum_{i=2}^k\sum_{j=1}^{i-1}z_iq^{(m-j)(i-1)+\binom{i}{2}}F_{m-j}^k({\bf z};q)F_{n-i+j}^k({\bf zq^{(i-1)(m+i-j)}};q).\\
F_n^k({\bf z};q)&=&F_n^{k-1}({\bf z};q)+\sum_{j=0}^{n-k}z_kq^{j(k-1)+\binom{k}{2}}F_j^{k-1}({\bf z};q)F_{n-k-j}^k({\bf zq^{(i-1)(k+j)}};q).\\
\det \hat{F}_{{\bf u},{\bf v}}&=&\left\{\begin{array}{lc}z_k^{n+k-1}q^{(k-1)\binom{n+k-1}{2}+\binom{k}{2}(n+k-1)}&\mbox{ if } k \mbox{ is odd,}\\ (-1)^{n-1}z_k^{n+k-1}q^{(k-1)\binom{n+k-1}{2}+\binom{k}{2}(n+k-1)}&\mbox{ if } k \mbox{ is even.}\end{array}\right.
\end{eqnarray*}

\subsubsection{$PRLP_n^k$}

Let $\pi=\pi_1\pi_2\dots\pi_m\in PRLP_n^k$, where $\pi_j$ is a layer for $1\leq j\leq m$.  Suppose $\ell(\pi_j)=i$.  Then the first $i-1$ elements of $\pi_j$ are decreasing.  So the first $i-2$ locations of $\pi_j$ are the tops of descents.  The last element in each layer in this case is larger than the first element in the next layer.  Thus, we combine the results of the distribution of the major index over $LP_n^k$ and $RLP_n^k$.  Because the element preceding the first element of $\pi_j$ is larger than the first element of $\pi_j$, part of the weight of $\pi_j$ is $\sigma-1$, i.e. the length of the permutation preceding $\pi_j$.  The descents in the first $i-2$ positions of $\pi_j$ contribute $(i-2)(\sigma-1)+\binom{i-1}{2}$.  Thus, we have that $wt(\pi_j)=z_iq^{(\sigma-1)(i-1)+\binom{i-1}{2}}$.  This gives us that $wt(\pi)=\left(\prod_{j=1}^mz_{\ell(\pi_j)}\right)q^{\maj(\pi)}$.  We define $$F_n^k({\bf z};q)=\sum_{\pi\in PRLP_n^k}wt(\pi).$$

We need only determine the shift factors $s_m^-(q)$ and $s_m^+(q)$.  Again, combining the results from $LP_n^k$ and $RLP_n^k$ we have that $s_{m,i}^+(q)=1$ and $s_{m,i}^-(q)=q^{(i-1)m}$.  We will write ${\bf zq^{(i-1)m}}$ for $(z_1,z_2q^{m},z_3q^{2m},\dots,z_kq^{(k-1)m})$.  

For $n\geq1$, the $q$-analogue of the $k$-Fibonacci numbers for the distribution of the major index over $PRLP_n^k$ satisfies $$F_n^k({\bf z};q)=z_1F_{n-1}^k({\bf zq^{(i-1)1}};q)+z_2F_{n-2}({\bf zq^{(i-1)2}};q)+\cdots+z_kq^{\binom{k-1}{2}}F_{n-k}({\bf zq^{(i-1)k}};q),$$ where $F_0^k({\bf z};q)=1$ and $F_n^k({\bf z};q)=0$ if $n<0$.  

Given the weight and shifting factor associated to the distribution of the major index over $PRLP_n^k$ we obtain the following identities for ${\bf u}$ and ${\bf v}$ as in \eqref{bob}:

\begin{eqnarray*}
F_{m+n}({\bf z};q)&=&F_m^k({\bf z};q)F_n^k({\bf zq^{(i-1)m}};q)+\\&&\sum_{i=2}^k\sum_{j=1}^{i-1}z_iq^{(m-j)(i-1)+\binom{i-1}{2}}F_{m-j}^k({\bf z};q)F_{n-i+j}^k({\bf zq^{(i-1)(m+i-j)}};q).\\
F_n^k({\bf z};q)&=&F_n^{k-1}({\bf z};q)+\sum_{j=0}^{n-k}z_kq^{(k-1)(j)+\binom{k-1}{2}}F_j^{k-1}({\bf z};q)F_{n-k-j}^k({\bf zq^{(i-1)(k+j)}};q).\\
\det \hat{F}_{{\bf u},{\bf v}}&=&\left\{\begin{array}{lc}z_k^{n+k-1}q^{(k-1)\binom{n+k-1}{2}+\binom{k-1}{2}(n+k-1)}&\mbox{ if } k \mbox{ is odd,}\\ (-1)^{n-1}z_k^{n+k-1}q^{(k-1)\binom{n+k-1}{2}+\binom{k-1}{2}(n+k-1)}&\mbox{ if } k \mbox{ is even.}\end{array}\right.
\end{eqnarray*}

\section{Set Partitions}

We now consider distributions of set partition statistics over certain set partitions that are counted by $k$-Fibonacci numbers.  Again these sets arise out of the theory of pattern avoidance.  We will describe these sets,  but not explain pattern avoidance in set partitions.  We refer the reader to~\cite{Goytpartitions3,Saganpartitionpatterns} where these sets can be found.  

A partition of $[n]=\{1,2,\dots,n\}$ is a family of disjoint subsets, $B_1,B_2,\dots,B_k$, of $[n]$ called {\it blocks} such that $\bigcup_{i=1}^k B_i=[n]$.  We will write $\pi=B_1/B_2/\dots/B_k$, where $\min B_1<\min B_2<\cdots<\min B_k$.  For example, $126/35/4/7$ is a partition of $[7]$.  

We say that a partition $\pi$ of $[n]$ is {\it layered} if $$\pi=12\dots i/(i+1)(i+2)\dots j/\dots/m(m+1)\dots n.$$  For example, $12/3/4567/89$ is a layered partition of $[9]$.  Let $L\Pi_n$ be the set of layered partitions, and let $L\Pi_n^k$ be the set of layered partitions with layer lengths at most $k$.  For example, $L\Pi_4=\{1234,1/234,12/34,123/4,1/2/34,1/23/4,12/3/4,1/2/3/4\}$, and $L\Pi_4^2=\{12/34,1/2/34,1/23/4,12/3/4,1/2/3/4\}$.  In~\cite{SagGoytqFib} the first author and Sagan study distributions of set partition statistics over $L\Pi_n^2$.  

As with layered permutations, there is an obvious bijection between $L\Pi_n^k$ and $T_n^k$.  As we did in Section 3, we will determine weights on layers by looking at distributions of set partition statistics on layered set partitions.  This will allow us to again determine $q$-analogues of the $k$-Fibonacci identities from Section 2.  

\subsection{The Right Bigger Statistic}

Wachs and White~\cite{WachsWhite} studied the distribution of the right bigger statistic and the left smaller statistic and determined $q$-analogues of the Stirling numbers of the second kind.  We will focus on the distributions of these two statistics on $L\Pi_n^k$.  

Let $\pi=B_1/B_2/\dots/B_m$ be a set partition.  For each element $b\in B_i$ with $i<j$, we have that $(b,B_j)$ is a {\it right bigger pair} if $b<\max B_j$.  Let  $rb(\pi)$ be the number of right bigger pairs in $\pi$.  

Again, let $\pi=B_1/B_2/\dots/B_m$ be a set partition.  For each element $b\in B_j$ with $j>i$, we have that $(b,B_i)$ is a {\it left smaller pair} if $b>\min B_i$.  Let $ls(\pi)$ be the number of left smaller pairs in $\pi$.  

The first author and Sagan~\cite{SagGoytqFib} show that $ls$ and $rb$ are equidistributed over layered set partitions with restricted block sizes.  Thus, we need only determine the identities for one of these statistics.  We will proceed to work with the $rb$ statistic.  

Notice that if $\pi=B_1/B_2/\dots/B_m\in L\Pi_n^k$ then every element in block $B_j$ is greater than every element in block $B_i$ if $i<j$.  Thus, the $\max B_j >b$ for each $b\in B_i$ with $i<j$. Now suppose that $|B_j|=i$, we will say the length of block $B_j$ is $i$.  In this case $wt(B_j)=z_if_{i,\sigma,\tau}(q)=z_iq^{\sigma-1}$, where $\sigma-1$ is the length of the partition preceding block $B_j$.  This shows that for $\pi=B_1/B_2/\dots/B_m\in L\Pi_n^k$, we have $wt(\pi)=\left(\prod_{j=1}^mz_{\ell(B_j)}\right)q^{rb(\pi)}$.  We define $$F_n^k({\bf z};q)=\sum_{\pi\in L\Pi_n^k}wt(\pi),$$ where $F_0({\bf z};q)=1$ and $F_n({\bf z};q)=0$ for $n<0$.  

The shift factors related to the distribution of the $rb$ statistic are $s_{m,i}^+(q)=1$ and $s_{m,i}^-(q)=q^m$.  The values of the shift vectors and the fact that $f_{i,\sigma,\tau}(q)=z_iq^{\sigma-1}$ give us that the distribution of the $rb$ statistic over $L\Pi_n^k$ is the same as the distribution of $\maj$ over $LP_n^k$.  Thus the results from Section 3.2.1 hold here.

\end{document}